\documentclass[10pt]{article}

\usepackage{enumerate}

\usepackage{amssymb}
\usepackage{amsthm}
\usepackage{amsmath}
\usepackage{graphicx}
\usepackage{extarrows}

\usepackage{endnotes}

\usepackage{color}
\usepackage{comment}

\allowdisplaybreaks

\newtheorem{theorem}{Theorem}

\newtheorem{lemma}[theorem]{Lemma}
\newtheorem{conj}[theorem]{Conjecture}

\newtheorem*{claim*}{Claim}

\newtheorem{problem}{Problem}

\newtheorem{prop}[theorem]{Proposition}

\theoremstyle{definition}

\theoremstyle{remark}

\def\eps{\varepsilon}

\def\F{\mathcal{F}}

\def\R{\mathbb{R}}
\def\Z{\mathbb{Z}}

\def\Xt{\widetilde{X}}
\def\St{\widetilde{S}}

\def\P{{\mathbb P}}
\def\E{{\mathbb E}}

\newcommand{\one}{\mathbf{1}}
	
\def\win{G}
\def\Cox{\hfill \Box}

\def\ee{\varepsilon}

\def\diseq{\stackrel {{\cal D}} {=}}
\def\|{\, | \, }

\begin{document}
	
	\begin{center}
		{\large \bf Success probability for selectively neutral 
			invading species in the line model {with a random fitness landscape}} 
	\end{center}
	\vfill
	
	\begin{flushright}
		Suzan Farhang-Sardroodi\footnote{Department of Mathematics,  
			Ryerson University, Toronto, Ontario M5B 2K3, Canada}, 
		Natalia L. Komarova\footnote{Department of Mathematics, University of
			California Irvine, Irvine, CA 92697; partially supported by NSF grant DMS-1812601}, 
		Marcus Michelen\footnote{Department of Mathematics, Statistics and Computer Science, University of Illinois 
			at Chicago, 
			851 S. Morgan Street Chicago, IL 60607-7045, Corresponding author} and 
		Robin Pemantle\footnote{Dept. of Mathematics, University of Pennsylvania, 
			209 South 33rd Street, Philadelphia, PA 19104.  Partially supported by 
			NSF grant DMS-1612674} \\
		{\tt suzan.farhang@ryerson.ca \\ 
			komarova@uci.edu \\ 
			michelen.math@gmail.com \\ 
			pemantle@math.upenn.edu}
	\end{flushright}
	\vfill
	
	\noindent{\sc Abstract:} \\[1ex]
	We consider a spatial (line) model for invasion of a population by a single mutant
	with a stochastically selectively neutral fitness landscape, independent
	from the fitness landscape for non-mutants.  This model is similar to those
	considered in~\cite{farhang2017effect, farhang2019environmental}.
	We show that the probability 
	of mutant fixation in a population of size $N$, {starting from a single mutant,} is greater than
	$1/N$, which would be the case  if there were no variation in fitness
	whatso{e}ver.  In the small variation regime, we recover precise asymptotics
	for the success probability of the mutant. This demonstrates 
	that the introduction of randomness provides an advantage to
	{minority} mutations in this model, and shows that the advantage increases with the system size. {We further demonstrate that the mutants  have an advantage in this setting only because they are better at exploiting unusually favorable environments when they arise, and not because they are any better at exploiting pockets of favorability in an environment that is selectively neutral overall.} 
	\vfill
	\vfill
	\vfill
	
	\\
	
	\\[1ex]
	
	\clearpage
	
	
	\section{Introduction}
	
	Evolution in  random environments has attracted attention of ecologists and mathematical biologists for a long time. Consider {direct} competition dynamics between two types of organisms whose reproduction and death rates may be different in different spatial locations. It is clear that organisms with larger reproduction rates and lower death rates are more likely to rise from low numbers and eventually replace their slow reproducing, rapidly dying counterparts. The situation becomes more complicated if the environment consists of different patches, where different types enjoy evolutionary advantage while others are suppressed. Depending on the properties of this patchy environment, the reproduction and death rates of the organisms, and the details of the evolutionary process, a number of outcomes can be observed, see e.g. \cite{chesson1981environmental, pulliam1988sources, hassell1994species, hanski2004ecology}, and also Modern Coexistence Theory \cite{ellner2019expanded}.
	
	{From early works of Haldane \cite{haldane1927mathematical}, Fisher \cite{fisher1930evolution} and Wright \cite{wright1931evolution} almost 100 years ago, an important focus of many theoretical studies of evolution has been the probability and timing of mutant fixation, see also Kimura's studies of neutral evolution \cite{kimura1968evolutionary, kimura1989neutral}. The general setting assumes the coexistence of different variants of an organism in a population, one of which is referred to as the ``wild type" (or ``normal"), and the other(s) as ``mutants" (or variants). Mutations may or may not confer selective advantage or disadvantage to an organism. In general, the term ``neutral' in evolutionary theory refers to the type of variants that, although different  the wild type, is neither advantageous nor disadvantageous, that is, it does not experience a positive or negative selection pressure.}
	
	{Mutant evolution in {\it random environments} became a topic of mathematical investigation around 1960s. Many early papers studied temporal fluctuations of the environment. For example, in \cite{gillespie1977natural}, it was assumed that while the wild types had constant numbers of offspring, mutants’ numbers of offspring were randomly changing every time step (but had the same mean as the wild types' offspring numbers). It was found that despite having the same mean number of offspring,  the mutants behaved as if they were disadvantageous. References \cite{frank1990evolution, frank2011natural} studied a more general setting, where the division rates of both wild types and mutants were  affected by the environmental changes. It was found that, surprisingly, the mutants behaved as if they were advantageous, despite having the same mean division rate, but only if the mutants were initially a minority. A similar result was found by  \cite{melbinger2015impact, cvijovic2015fate}. Many results have been obtained in the framework of the Modern Coexistence Theory in ecology, e.g. regarding the instantaneous rate of increase of a rare species \cite{chesson1994multispecies, chesson2000mechanisms, adler2007niche}. It was shown analytically by \cite{chesson1981environmental,meyer2018noise,meyer2020evolutionary} that temporal randomness in division rates  leads to a positive rate of increase of a minority mutant. Another set of analytical results concerns  extinction times  \cite{kessler2015neutral,hidalgo2017species, danino2018fixation}.}
	
	{In contrast to temporal variations, spatial environmental variations are associated with fitness differences that characterize different spatial locations (and do not change in time). For example, one can consider a stylized model where light conditions differ in different locations, and therefore growth and reproduction properties of plants may differ  spot to spot. Let us suppose that the wild type plant needs high light to grow, but a mutant prefers shade. Then  spots characterized by strong lighting conditions will result in an increase in wild type growth rate and a decrease in mutant growth rate. What can we say about the mutant fixation probability if the ``high light" and ``low light" spots are distributed with equal likelihood? In this example, the fitness values of wild type and mutant organisms are anti-correlated, that is, in a given spot, if a wild type plant has an elevated fitness value, a mutant will have a reduced fitness value. Different scenarios are possible, including the case where fitness values of wild type and mutant organisms are uncorrelated; this would correspond to a situation where the growth properties of wild type and mutant plants are determined by different and uncorrelated environmental factors, such as light and nutrients.}
	
	Two {important} examples of  biological systems  where evolution takes place in the presence of spatial randomness, are biofilms and tumors.  Biofilms are collectives of microorganisms, such as bacteria or fungi, that coexist on surfaces within a slimy extracellular matrix. Evolutionary dynamics of these microorganisms take place in an environment characterized by significant heterogeneities, both in physical and chemical parameters, such as heterogeneities in the  interstitial fluid velocity, gradients in the distribution of nutrients and  other metabolic substrates/products \cite{stewart2008physiological, jayasinghe2014metabolic}. It has been suggested \cite{boles2004self}  that different organisms may respond differently to these diverse environmental stimuli, giving rise to evolutionary co-dynamics that can be modeled by using models similar to those studied here. The second example is evolution in cancerous populations, where the presence of highly heterogeneous environments has been documented, see e.g. \cite{li2007visualization, graves2010tumor}. Cancerous cells in different locations across a tumor are exposed to different concentrations of  oxygen, nutrients, immune signaling molecules, inflammatory mediators, and other non-malignant cells that comprise the tumor microenvironment. Understanding tumor evolution under these spatially heterogeneous conditions is essential for  understanding and combating long-standing challenges in oncology such as drug resistance in tumors. It also presents opportunities for creating new therapeutic strategies \cite{yuan2016spatial}.

	{In the literature, several modeling approaches have been used to study spatial randomness. In one class of models, agents are placed  on a random network, where different vertices have different degrees; the nodes'  fitness values are based on their numbers of interactions, making some vertices more advantageous than others. These types of settings have been used e.g. in the context of the game theory/cooperation (e.g. \cite{santos2005scale, santos2006graph, santos2006cooperation, santos2008social, tomassini2007social, maciejewski2014evolutionary}). Another class of models is a finite island model, where agents are placed in patches (characterized by environmental differences) and a certain degree of patch-to-patch migration is assumed. Mutant fixation probability has been studied in the high migration rate \cite{nagylaki1980strong} and the low migration rate \cite{tachida1991fixation} limit. Mutant fixation probability in the problem with two patches has been solved analytically in \cite{gavrilets2002fixation}, where it was assumed that the mutation is advantageous in one patch and deleterious in the other patch.  An extension  to a multiple patch model was provided in \cite{whitlock2005probability}, who investigated the accuracy of various approximations for mutant fixation probability. The role of spatially variable environments has been also addressed by the Modern Coexistence Theory, see e.g. studies of species coexistence in \cite{chesson2000general}.}
	
	{In the recent papers \cite{farhang2017effect, farhang2019environmental} we studied the dynamics of mutant fixation in a model that is a generalization of the classical Moran model \cite{moran1958random} and includes spatial randomness. }  We assumed that the population of organisms (or agents)  remains constant and birth/death updates are performed with rules governed by the organisms' fitness parameters (birth and/or death rates). Interactions of replacing dead organisms by offspring of others happen along edges of a network that defines ``neighborhoods". For example,  in a model characterized by agents on a complete graph, every agent is in the neighborhood of everyone else, and therefore a dead organism can be replaced by offspring of any other agent. On the other hand, on a circular graph, each agent has exactly two neighbors. It was assumed that, for each realization of the evolutionary competition process, for each of the $N$ sites, the birth and/or death rates of both types were assigned by randomly drawing  the same distributions of values. Then the probability of mutant fixation, starting  a given initial location of mutant agents among the $N$ spots, was calculated. Finally, this probability was averaged over all realizations of the fitness values. It was found that, somewhat surprisingly, the mutants showed an advantage compared to the normal types, as long as their initial number was smaller than a half. This result can be obtained for particular (relatively small) numbers of $N$, but no asymptotic results for large values of $N$ were obtained analytically. {It was observed, however, that the effect of randomness to ``favor" minority mutant increased with the system size.}

	In this paper, we {focus on the asymptotic behavior of the fixation probability of mutants in the presence of spatial randomness. We consider a spatial model similar to that used in \cite{farhang2017effect, farhang2019environmental}. It is a spatial (1D) version of the Moran process (see e.g. \cite{komarova2006spatial}) where spatial variations in the environment are implemented by random fitness values of wild type and mutant individuals at different sites. Models of this type (but without random fitness values) have been used previously to study cellular evolution in the context of cancerous transformation (\cite{michor2004linear, komarova2006spatial}) and are relevant for describing e.g. colonic crypts. To the best of our knowledge the results reported here are the first rigorous results for the problems of this kind.}

	\section{{Model formulation} and results}

	We consider the following model.  The spatial environment consists of $N$ sites, numbered $1, \ldots , N$
	arranged in a line { with nearest-neighbor edges}.  At each site there are two real parameters 
	representing \emph{fitness} values: a mutant 
	fitness and a normal fitness, each chosen IID $1 \pm \delta$.  
	These fitness values will remain fixed while the \emph{state} of each site 
	will change.  Site $1$ begins with state ``mutant'' and all other sites
	begin with the state ``normal.''  The evolution proceeds { in discrete time} 
	as follows: { replace each edge with two directed edges, one in each direction; 
		at each time-step  choose a directed edge $(j,k)$ with $|j - k| = 1$ uniformly at random, and let $\nu_k$ and $\mu_k$ be the normal and mutant fitnesses of $k$; if $j$ is mutant, then we set $k$ to be mutant with probability $\mu_k/(\nu_k + \mu_k)$ and leave $k$ unchanged with the remaining probability; similarly, if $j$ is normal
		then we set $k$ to be normal with probability $\nu_k/(\nu_k + \mu_k)$ and leave it unchanged otherwise.
		
		This model may also be thought of as occurring in continuous time: 
		Each directed edge is
		assigned an exponential clock of rate $1$.  When the clock  edge $u$ to $v$ rings, 
		$u$ attempts to replace the type of $v$ with its own type; if $u$ is mutant, then the state of 
		$v$ is set to be mutant with $\frac{\mu_v}{\mu_v + \nu_v}$ and is unchanged with 
		the remaining probability.  
	}
	
	Since there are only finitely many sites and only two types, the process eventually 
	fixates  in one of two states: all mutants or all normal.  
	We are interested in the probability of the event $\win$ of fixating in the state 
	where all sites are mutants, and in particular  how the probability that 
	$\win$ occurs changes---
	after averaging over the random environment---as $\delta$ varies.  More 
	concretely: should more or less randomness help the mutant dominate?
	
	If $\delta = 0$, there is no differential fitness  and the fitness
	environment is deterministic.  After $k$ replacements,
	the mutants will always either be extinct or occupy some interval 
	$1, \ldots, X_k$.  The process $\{ X_k \}$ is a simple random walk
	stopped when it hits $0$ or $N$, hence the probability that it stops 
	at $N$ is precisely $1/N$.  {Biologically this means that in the absence of any fitness differences between the wild type and mutant cells, the probability of any cell to fixate is the same and equals $1/N$. Note that if fixation probability is greater (smaller) than the initial share of the mutant, then this indicates the presence of positive (negative) selection acting in the system. }

	In the model considered here, when $\delta > 0$, the dynamics become more complicated.  In fact they are 
	the dynamics of a birth-death process in a random environment; 
	equivalently, the dynamics may be thought of as a variant of the voter
	model where each site may be more or less susceptible to a given
	type.  A similar model, but with circular boundary conditions 
	($1$ and $N$ are neighbors), was analyzed in~\cite{farhang2017effect}.  
	There, it was proved for $4 \leq N \leq 8$ and empirically observed for much
	larger values of $N$ that the probability of a mutant takeover 
	is strictly greater than $1/N$, {indicating the presence of positive selection for the mutant, although its fitness values are chosen  the same distribution as those for the wild type cells.}  The goal of this paper is to 
	establish the analogous result rigorously for the line model and 
	to give precise asymptotics for the  annealed probability of a
	mutant takeover.\\

	Let $(\Omega_N , \F_N , \P_N)$ be a probability space on which are defined
	independent Rademacher random variables (that is $\pm 1$ fair coin flips)
	$B_1, \ldots , B_N$ and $B_1', \ldots , B_N'$, as well as rate~1 Poisson
	processes $\xi_t^{(i,j)}$ for $1 \leq i,j \leq N$ and $|i-j| = 1$, 
	independent of the Rademacher variables and of each other.  For $\delta 
	\in (0,1)$, the normal  fitness at site $k$ is the quantity 
	$\mu_k := 1 + \delta B_k$ and the mutant fitness at site $k$ is the 
	quantity $\nu_k := 1 + \delta B_k'$; in this way, the model is defined
	simultaneously for all $\delta$, although we will not do much to exploit 
	this simultaneous coupling.
	
	The states of the process are configurations where each site has 
	a mutant (one) or normal cell (zero).  { Since we always consider
		the starting condition of having one mutant at site $1$ and all others are normal,
		the collection of mutant cells is always some}
	segment of sites $[1,k]$ and normal cells thereafter.  Hence 
	we can identify the state space with $\{ 0 , 1, \ldots , N \}$,
	with $0$ corresponding to mutant extinction.  { Since we need only
		keep track of the right-most mutant to describe the state of the process,
		we first find the transition probabilities for the evolution of this right-most point.  Figure \ref{fig:model} shows an instance of the model along with this identification.}
	
	\begin{figure}[!ht]
		\centering
		\includegraphics[width=6in]{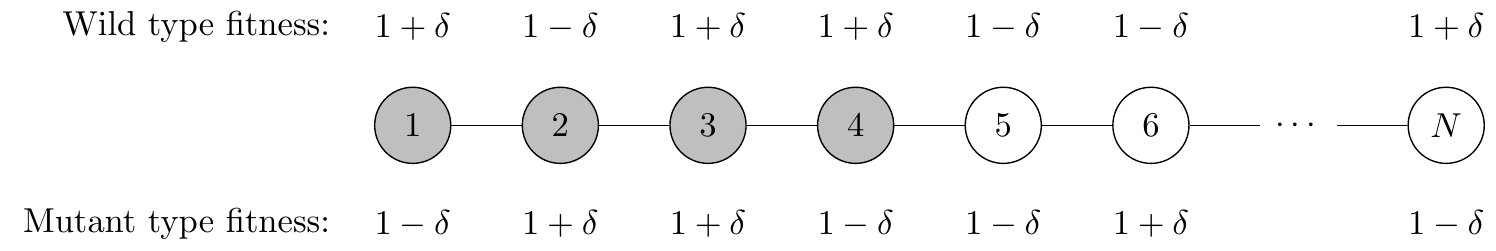}
		\caption{An instance of the model with mutant sites in gray and normal sites in white and both mutant and wild fitness types listed. The above state is identified with $4$, since the mutant sites are $\{1,2,3,4\}$.}
		\label{fig:model}
	\end{figure}
	
	At times { corresponding to points of the Poisson process }
	$\xi^{(i,j)}$, cell $i$ tries to reproduce at site $j$. 
	This only matters if $i=k$ or $j=k$, { since otherwise sites $i$ and $j$
		have the same state and no change in state can occur.}
	Sampling only when the configuration changes yields a discrete time birth and death
	chain, absorbed at 0 and $N$, whose transition probabilities 
	are easily characterized.  Define the random quantities
	\begin{equation} \label{eq:beta}
	\beta_k := \frac{\mu_k}{\nu_k + \mu_k} \, .
	\end{equation}
	From state $k$ the only relevant directed edges are $(k,k+1)$
	and $(k+1,k)$ { since these corresponds to the mutant site $k$
		making $k+1$ mutant and normal site $k+1$ making $k$ normal.}
	Both attempted at rate~1 and succeeding with
	respective probabilities $\beta_{k+1}$ and $1 - \beta_k$. {
		Letting $p_k$ denote the transition probability  the right-most mutant
		being $k$ to being $k+1$}, we have 
	\begin{equation} \label{eq:p_k}
	p_k = \frac{\beta_{k+1}}{\beta_{k+1} + (1 - \beta_k)}  \, .
	\end{equation}
	
	{ We may now think of the evolution as occurring entirely
		on $\{0,1,\ldots,N\}$ where state $k$ moves to step $k+1$ with probability
		$p_k$ and moves to $k-1$ with probability $1 - p_k$.
	}
	
	Let $\win = \win(\delta)$ denote the event that the absorbing 
	state $N$ is reached before the absorbing state~0, under dynamics
	for the given $\delta$.  Our first result is an asymptotic expression 
	for $\P_N (\win(\delta))$ in the regime where $N \to \infty$ and
	$\delta \sqrt{N} \to c$.
	\begin{theorem}[asymptotics when $\delta \sqrt{N} \to c$] \label{th:main}
		Fix $c > 0$ and suppose $N \to \infty$ and $\delta \sqrt{N} \to c$.  Then
		\begin{equation} \label{eq:main}
		N \P_N(\win(\delta)) \to g(c)
		\end{equation}
		where 
		$$g(c) =  \E \left[ \frac{1}{\int_{0}^1 \exp(\sqrt{2} c B_s) \,ds } 
		\right]$$
		for a standard Brownian motion $\{ B_s \}$.  
		The function $g$ is continuous and strictly increasing on $(0,\infty)$.  
		It satisfies 
		\begin{eqnarray} 
		g(c) - 1 & \sim & \frac{c^2}{{6}} \qquad \mbox{ as } c \downarrow 0 \, ;
		\label{eq:c=0} \\
		g(c) & \sim & \frac{c}{\sqrt{\pi}} \qquad \mbox{ as } c \to \infty \, .
		\label{eq:c=infty}
		\end{eqnarray}
	\end{theorem}
	
	We note that continuity of $g$ implies that for $\delta \ll N^{-1/2}$, then $\P_N(G(\delta)) \sim 1/N$ as in the $\delta = 0$ case. 
	In the regime where $\delta \gg N^{-1/2}$ but still $\delta \ll 
	(\log N)^{-\ee}$, the asymptotic behavior of $\P_N (G(\delta))$
	is as follows.
	
	\begin{theorem} \label{th:better} Assuming $\delta \sqrt{N} \to \infty$, suppose that there
		is an $\ee > 0$ such that $\delta (\log N)^\ee \to 0$.  Then
		$$\P_N (G(\delta)) \sim \frac{\delta}{\sqrt{\pi N}} \, .$$
	\end{theorem}
	
	We do not expect this to hold if $\delta = \Theta (1)$ as $N \to \infty$ 
	because without scaling, the graininess of the random walk may lead
	to a different constant than would be obtained by a Brownian approximation.
	Nevertheless, we believe the condition $\delta = o(\log N)^{-\ee}$ to
	be unnecessary and we conjecture the following.
	\begin{conj} \label{conj:1}
		If $\delta \sqrt{N} \to \infty$ and $\delta \to 0$ then 
		$$\frac{\P_N (G(\delta))}{\delta N^{-1/2}} \to \frac{1}{\sqrt{\pi}} \, .$$
		In any case, as $N \to \infty$ and in the absence of the requirement  $\delta\to 0$, 
		$$0 < \frac{C_1 \delta}{N^{1/2}} \leq \P_N (G(\delta)) 
		\leq \frac{C_2 \delta}{N^{1/2}} \, .$$
	\end{conj}
	
	We interpret Theorems~\ref{th:main} and~\ref{th:better} as saying that 
	the stochastic environment favors a minority invader. {Indeed, in the absence of any randomness (that is, $\delta=0$), the probability of neutral mutant fixation on a circle is given by $1/N$ (a result that can be demonstrated e.g. by simple symmetry considerations). Mutant fixation on a line model similar (but not identical) to the present one was studied by \cite{komarova2006spatial} and it was shown that it depends on initial the location of the mutant. It is the smallest for a mutant originally located at one of the ends of a line and increases toward the middle initial location, but never exceeds the value $1/N$. In the present model, in the absence of randomness, mutant fixation probability is given by $1/N$. Theorems ~\ref{th:main} and~\ref{th:better} state that mutant fixation probability in the presence of randomness is greater than $1/N$, and that the quantity $N\P_N (G(\delta))$ increases with the system size ($N$) and with the amount of randomness ($\delta$). In other words, despite having no explicit advantage, a mutant in the random environment gets fixated with a probability that is significantly larger than in the case of a non-random environment.  }

	The following
	result shows that this effect is due to the minority taking
	advantage of the cases where the overall environment 
	is more favorable, not environments where pockets 
	favoring each type appear but are balanced against each other.
	
	\begin{theorem} \label{th:ident}
		Let $N = 2k$ be an even integer and let $Q_N$ denote $\P_N$
		conditioned on $\sum_j B_j = \sum_j B_j'$.  Then
		$N Q_N (G(\delta)) = 1$ for all $N$ and all $\delta$.
	\end{theorem}
	
	{To rephrase in biological terms, we note that among different realizations of wild type and mutant fitness values, there are cases where mutants experience an overall advantage ($\sum_jB_j<\sum_jB'_j$), an overall disadvantage ($\sum_jB_j>\sum_jB'_j$), or have a fitness configuration whose net sum is equal to that of the wild types, although locally mutants may experience positive or negative selection pressure (the case $\sum_jB_j=\sum_jB'_j$). Theorem \ref{th:ident} states that if we only consider the latter type of environments, mutants will behave exactly as expected in the absence of randomness.  On the other hand, configurations with a net mutant advantage and disadvantage do not balance each other out and result in a positive selection pressure experienced by the mutant.}

	The outline of the remainder of the paper is as follows.  
	In the next section we show how the computation of $\P_N (G(\delta))$
	reduces to computing an expectation of a functional of a random walk. 
	{From here, Theorems \ref{th:main} and \ref{th:better} can heuristically be inferred  replacing the random walk with a corresponding Brownian motion via Donsker's Theorem.  However, the only regime in which Donsker's Theorem applies is that of Theorem \ref{th:main}.}
	{ Using this approach, }we then verify in the case $\delta \sim c N^{-1/2}$ that the expectation 
	commutes with the Brownian scaling limit.  Section~\ref{sec:brownian} 
	computes the corresponding expectations for Brownian motion, based
	on results of Matsumoto, Yor and others.  Section~\ref{sec:proof}
	puts this together to prove Theorem~\ref{th:main}.  We also give the
	relatively brief proof of Theorem~\ref{th:ident}.  Theorem~\ref{th:better}
	is proved in Section~\ref{sec:better}.  This is proved in two stages,
	first when $\delta$ is required to decrease more rapidly than $(\log N)^{-1}$
	and then when this is relaxed to $(\log N)^{-\ee}$.  The final section
	presents some numerical simulations and further questions.
	
	\section{A scaling result} \label{sec:scaling}
	
	The following explicit formula for the probability of a birth and 
	death process started at~1 to reach $N$ before~0 is well known; {we include its short proof for completeness.} 
	
	\begin{prop} \label{pr:BD}
		In a birth and death process, 
		let $p_k$ be the probability of transition to $k+1$ from $k$ and
		let $q_k := 1 - p_k$ be the probability of transition to $k-1$.  
		Let $Q_x$ denote the law of the process starting  $x$ and $\tau_a$
		the hitting time at state $a$.  Then
		\begin{equation} \label{eq:BD}
		Q_1 (\tau_N < \tau_0) 
		= \frac{1}{\sum_{k=0}^{N-1} \prod_{{ j}=1}^k \frac{q_j}{p_j}} \, .
		\end{equation}
		Here, the first term of the sum is the empty product, 
		equal to~1 by convention.
	\end{prop}
	\begin{proof}
		Consider the network in which the resistance between $k$ and $k+1$ is $\prod_{j = 1}^k (q_j/p_j)$.  Note then that the random walk on this network is equivalent to that described in the Proposition.  The expression for $Q_1$ is the 
		ratio of the conductance  $1$ to $N$ to that plus the conductance
		$1$ to $0$.  \end{proof}
	
	We now show that the denominator is close to a functional of a random 
	walk, which is close to a functional of a Brownian motion, and that
	these approximations are good enough to pass expectations to the limit.
	
	For $1 \leq k \leq N-1$ denote
	\begin{eqnarray*}
		X_k & := & \log \frac{q_k}{p_k} = \log \frac{1 - \beta_k}{\beta_{k+1}} \\
		\Xt_k & := & \log \frac{1 - \beta_{k+1}}{\beta_{k+1}} 
		= \log \frac{\nu_{k+1}}{\mu_{k+1}} 
	\end{eqnarray*}
	with partial sums $S_k := \sum_{j=1}^k X_j$ and likewise for $\St$.
	{ Define $S_0 = \St_0 = 0$.}
	The definition of $X_k$ is chosen so that Equation~\eqref{eq:BD} 
	becomes 
	\begin{equation} \label{eq:approx}
	\P_N (\win(\delta)) = \E \frac{1}{\sum_{k=0}^{N-1} \exp (S_k)} \, .
	\end{equation}
	On the other hand $\Xt_k$ are chosen so that $\{ \St_k \}$ is 
	precisely a simple random walk on the lattice $\delta' \Z$, 
	with holding probability $1/2$, where 
	$$\delta' := \log \frac{1+\delta}{1 - \delta} = 2 \delta + O(\delta^2) \, .$$
	{ For $\delta \leq 1 - \eps$, $\mu_j$ and $\nu_j$ are uniformly bounded away 
		$0$ and $1$ and so $\beta_j$ is bounded away  $1$ as well.  Thus }
	\begin{equation} \label{eq:delta}
	|\St_k - S_k| = |\log (1 - \beta_{k+1}) - \log (1 - \beta_1)|
	\leq C_\ee \delta
	\end{equation}
	as long as $\delta \leq 1 - \ee$ { by applying Taylor's theorem
		with remainder to $\log (1 - \beta_{k+1})$ and $\log(1 - \beta_1)$ as a function of $\delta$.} 
	Donsker's theorem { then} gives
	\begin{equation} \label{eq:donsker}
	(S_{\lfloor t N \rfloor})_{t \in (0,1)} \xrightarrow{N \to \infty} 
	\sqrt{2} \cdot c \cdot {(B(t))_{t \in (0,1)}}
	\end{equation}
	in the c\`adl\`ag topology whenever $\delta \sqrt{N} \to c$, 
	where $B(t)$ is Brownian motion.  
	
	In a moment we will show
	\begin{lemma} \label{lem:UI}
		Suppose $\delta$ and $N$ vary so that $\delta \sqrt{N}$
		remains bounded away  zero and infinity.  Then the random variables 
		$\{N  Q_1(\tau_N < \tau_0) \}$ are uniformly integrable.  {Further, if $\delta \sqrt{N} \to c$ then $$N \P_N(\win(\delta)) \to 
			\E \left[ \frac{1}{\int_{0}^1 \exp(\sqrt{2} c B_s) \,ds } \right]$$
			where $(B_s)$ is standard Brownian motion.}
	\end{lemma}
	
	{ The second half of Lemma \ref{lem:UI} is the first part of Theorem \ref{th:main} and follows  
		uniform integrability together with~\eqref{eq:BD} and~\eqref{eq:donsker}.  This is because
		convergence of means follows  
		uniform integrability together with convergence in distribution.}
	%

	\noindent{\sc Proof of Lemma}~\ref{lem:UI}: 
	{ A consequence of~\eqref{eq:delta} is that
		$$\left|\frac{\sum_{j = 0}^{N-1} \exp(\St_j)}{\sum_{j=0}^{N-1} 
			\exp(S_j)} -1 \right| = o(1) \, .$$ 
		It therefore suffices to show that the variables 
		$$\frac{N}{\sum_{j = 0}^{N-1} \exp(\St_j)}$$
		are uniformly integrable.}

	For a simple random walk, the reflection principle gives
	\begin{eqnarray*}
		\P[\min_{j \leq n} S_j \leq -r ] & = &
		\P[S_n = -r  ] + 2 \P[S_n < -r] \\
		& \leq & 2 \P[S_n \leq -r]   \\
		& \leq & { 2 \exp\left( - 2r^2 / n \right)}
	\end{eqnarray*}
	{ where the last bound is by Hoeffding's inequality.}
	Because $\{ \St_k \}$ is a simple random walk scaled by $\delta'$
	and holding with probability $1/2$, for all $k \in \Z^+$ and $t > 0$,
	\begin{equation} \label{eq:LD}
	\P[\min_{j \leq k} \St_{j} < -t] \leq 2 
	\exp\left(-\frac{ { 2}t^2}{\left(\log\left(\frac{1 + \delta}{1 - \delta} 
		\right)\right)^2 k } \right) \, .
	\end{equation}
	
	Applying~\eqref{eq:LD} shows that
	\begin{equation}\label{eq:walk-tail-bound}
	\P[\min_{k \leq N\eps} \St_k \leq -1 ] \leq 2 \exp\left(- B / \eps \right)
	\end{equation} 
	for $B$ depending continuously on $\delta \sqrt{N}$.  Thus { for $\eps \in (0,1]$}
	$$\frac{N}{\sum_{j = 0}^{N-1} \exp(\St_j)} \leq 
	\frac{N}{\sum_{j = 0}^{\eps N} \exp(\St_j) } \leq 
	\frac{1}{\eps \exp\left(\min_{j \leq \eps N} \St_j \right)}\,.$$  
	{ Thus, for $x$ large enough and picking $\eps = e/x$ we have $$\P\left[\frac{N}{\sum_{j = 0}^{N-1} \exp(\St_j)} > x \right] \leq \P\left[\frac{1}{\eps \exp\left(\min_{j \leq \eps N} \St_j \right)} > x \right] = \P\left[ \min_{j \leq (eN/x)}\St_j < -1 \right]\,.$$ }
	We have, for each { $K \geq e$}, 
	\begin{align*} \E\left[ \frac{N}{\sum_{j = 0}^{N-1}\exp(\St_j)  }  
	\one\left\{ \frac{N}{\sum_{j=0}^{N-1} \exp(\St_j) } {>} K \right\}  \right]
	&= \int_{x \geq K} \P\left[ \frac{N}{\sum_{j = 0}^{N-1} \exp(\St_j) } {>} x 
	\right] \,dx\\
	&\leq \int_{x \geq K} \P\left[ \min_{j \leq (e N / x)} \St_j  
	{<} -1 \right] \,dx \\
	&\leq \int_{x \geq K} 2 \exp(-C x / e) \,dx\,.
	\end{align*}
	This inequality holds for all $N$ and converges to zero as 
	$K \to \infty$, thereby showing uniform integrability.  
	$\Cox$
	
	\section{Evaluation of the Brownian integral} \label{sec:brownian}
	{ Define the following functions of Brownian motion}:
	\begin{eqnarray}
	A_\alpha (t) & := & \int_0^t e^{\alpha B_s} \, ds \label{def:At} \; ; \\
	m_\alpha (t) & := & \E A_\alpha (t)^{-1} \, .\label{def:m}
	\end{eqnarray}
	In this notation, Lemma~\ref{lem:UI} proves the first statement
	of Theorem~\ref{th:main} with
	\begin{equation} \label{eq:g} 
	g(c) := m_{c \sqrt{2}} (1) \, .
	\end{equation}
	To finish the proof of Theorem~\ref{th:main}, it remains to 
	evaluate~\eqref{eq:g}.  Expectations such as the one in~\eqref{def:m} 
	have been well studied.
	
	\begin{prop}[\protect{\cite{MatsumotoYor2005}}] \label{prop:MY}
		Let $\{ B_t : t \geq 0 \}$ be a standard Brownian motion and
		let $A(t) := \int_0^t e^{2 B_s} \, ds$.  Then,
		\begin{eqnarray}
		\E \left [ A_2 (t)^{-1} | B_t = x \right ] & = & \frac{x e^{-x}}{t \sinh x}
		\mbox{ if } x \neq 0 \label{eq:MY1} \, ; \\
		\E \left [ A_2 (t)^{-1} | B_t = 0 \right ] & = & t^{-1} \label{eq:MY2} \, ; \\
		m_2 (t) & \sim & \sqrt{\frac{2}{\pi t}} 
		\mbox{ as } t \to \infty \label{eq:MY3} \, .
		\end{eqnarray}
	\end{prop}

	The next lemma uses Brownian scaling to transfer these results 
	to the $-1$ moment of $\int_0^1 e^{\alpha B_s} \, ds$. 

	\begin{lemma} \label{lem:MY rescaled}
		For $\alpha , \nu , t > 0$,
		\begin{equation} \label{eq:alpha}
		m_\alpha (t) = \frac{\alpha^2}{\nu^2} 
		m_\nu \left ( \frac{\alpha^2}{\nu^2} t \right ) \, .
		\end{equation}
		It follows that
		\begin{equation} \label{eq:m}
		m_\alpha (1) \sim \frac{\alpha}{\sqrt{2 \pi}} \mbox{ as } 
		\alpha \to \infty \, .
		\end{equation}
	\end{lemma}
	
	{Both proofs are straightforward although somewhat technical, and so we defer them to Appendix \ref{sec:scaling-proofs}.}
	\section{Proofs of Theorems~\protect{\ref{th:main}} 
		and~\protect{\ref{th:ident}}} \label{sec:proof}
	
	\noindent{\sc Proof of Theorem~\ref{th:main}:}
	We have already evaluated $g$.  Continuity and strict monotonicity 
	will follow  {computing the second derivative of $\phi(x) = x e^{-x}/\sinh(x)$ explicitly}.  
	The estimate~\eqref{eq:c=infty} follows immediately ~\eqref{eq:g} 
	and~\eqref{eq:m}.  It remains to prove~\eqref{eq:c=0}, that is,
	to estimate $g(c) = m_{c \sqrt{2}} (1)$ near $c=0$.  
	Integrating~\eqref{eq:MY1} gives
	$$t \, m_2 (t) = \int \frac{x e^{-x}}{\sinh (x)} \, dN(0,t) (x) \, .$$
	Plugging in $x e^{-x} / \sinh (x) = 1 - x + x^2 / 3 + O(x^3)$
	gives
	$$t \, m_2 (t) = 1 + \frac{t}{3} + O(t^{3/2})$$
	as $t \downarrow 0$.  Using~\eqref{eq:alpha} with $\alpha = c \sqrt{2}$ 
	and $\nu = 2$ then gives
	$$g(c) = \frac{c^2}{2} m_2 \left ( \frac{c^2}{2} \right )
	= 1 + \frac{c^2 + o(1)}{6} , \, $$
	proving~\eqref{eq:c=0}.
	$\Cox$
	\noindent{\sc Proof:} We show that $\phi''(x) > 0$ for $x \neq 0$.
	Computing, $\phi''(x) = e^{-x} (\sinh (x))^{-3} h(x)$ where
	$$h(x) = x + 1 + (x-1) e^{2x} \, .$$
	Because we have taken out a factor of the same sign as $x$, we need to 
	show that $h$ is positive on $({0},\infty)$ and negative on $(-\infty , 0)$.  
	Verifying first that $h(0) = h'(0) = 0$, the proof is concluded by
	observing that $h'(x)$ has a unique minimum at $x=0$, because 
	$h''(x) = 4x e^{2x}$ has the same sign as $x$.
	$\Cox$
	
	\noindent{\sc Proof of Theorem~\ref{th:ident}:}
	Extend the definition of $X_k$ 
	by reducing modulo $N$, thus $X_N := \log (1 - \beta_N) - \log \beta_1$
	and so forth.  This makes the sequence $\{ X_k : k \geq 1 \}$ periodic
	and shift invariant, that is, $(X_1, \ldots , X_{N-1} , X_N) \diseq 
	(X_2 , \ldots , X_N , X_1)$.  Observe also that $Q_N (S_N = 0) = 1$ 
	because $S_N = \sum_{j=1}^N \log (1 - \beta_j) - \sum_{j=1}^N \log \beta_j$
	and the multiset of $N$ values of $\beta_j$ is the same as the multiset
	of $N$ values of $1 - \beta_j$.  { This implies that for each $k \geq 0$ we have $S_{k+N} = S_k$ and so $$\sum_{j = 0}^{N-1} \exp(S_j) = \sum_{j = 0}^{N-1} \exp(S_{j +k}),.$$  By shift invariance of $(X_1,\ldots,X_N)$, we may shift this sequence $k$ times to shift the sequence $(S_k,\ldots,S_{k+N})$ to $(S_{k} - S_k, \ldots, S_{k+N} - S_k)$\,.  In particular, this shows that $$ \E \frac{1}{\sum_{j = 0}^{N-1} \exp(S_j)} =\E  \frac{1}{\sum_{j=0}^{N-1} \exp(S_{j+k})} = \E  \frac{1}{\sum_{j=0}^{N-1} \exp(S_{j} - S_k)} = \E \frac{S_k}{\sum_{j =0}^{N-1} \exp(S_j)}\,. $$  Averaging over all $0 \leq k \leq N-1$ shows 
		$$
		\E \frac{1}{\sum_{j = 0}^{N-1} \exp(S_j)} = N^{-1} \sum_{k = 0}^{N-1}\E \frac{S_k}{\sum_{j =0}^{N-1} \exp(S_j)} = N^{-1}
		$$}
	%
	proving Theorem~\ref{th:ident}.
	$\Cox$
	
	\section{Proof of Theorem~\protect{\ref{th:better}}} \label{sec:better}
	
	\subsection{KMT Coupling and Preliminaries}
	
	A key result for studying this larger regime $\delta$ is coupling
	of random walk to Brownian motion.
	\begin{lemma}[KMT-Coupling] \label{lem:KMT}
		Let {$\{X_k\}_{k\geq 0}$} be a simple random walk with i.i.d. increments $\xi$ 
		so that $\E[\xi] = 0, \E[\xi^2] = 1$ and $\E[e^{t|\xi|}] < \infty$ 
		for $t$ sufficiently small.  Extend $X_k$ to continuous time by 
		defining $X_t = X_{\lfloor t \rfloor}$.  Then there exists a constant 
		$C$ so that for all $T > e$ there is a coupling so that 
		$$\P\left[\sup_{t \in [0,T]} |X_t - B_t| \geq C \log(T)  \right] 
		\leq \frac{1}{T^2}$$
		where $B_t$ is standard Brownian motion.
	\end{lemma}
	\begin{proof}
		This is stated as equation~(17) in~\cite{aurzada-dereich}.
		In fact one can replace $T^{-2}$ by a stretched exponential;
		a slightly weaker result that may be quoted 
		~\cite[Theorem~7.1.1]{lawler-limic} is that for any $\alpha > 0$
		there is a $C_\alpha$ such that 
		\begin{equation} \label{eq:stretched}
		\P\left[\sup_{t \in [0,T]} |X_t - B_t| \geq C_\alpha \log(T)  \right] 
		\leq T^{-\alpha} \ .
		\end{equation} 
	\end{proof}
	
	Another basic lemma is the well-known uniform estimate for random 
	walk hitting probabilities { whose proof is standard and so we prove it in
		Appendix \ref{app:better-proofs}.}
	
	\begin{lemma} \label{lem:RW}
		Let {$\{ S_n \}_{n \geq 0}$} be a random walk whose IID increments $\{ X_n \}$ 
		satisfy the hypotheses of the KMT coupling.  Then, if $u$ ranges 
		over $[M^\ee , M^{1/2 - \ee}]$ for some $\ee \in (0,1/2)$, 
		$$\P (\max_{1 \leq j \leq M} S_j \leq u) \sim \sqrt{\frac{2}{\pi}} 
		\frac{u}{M^{1/2}}$$
		as $M \to \infty$, uniformly in $u$ and the random walk.
	\end{lemma}
	We will prove Theorem~\ref{th:better} in two cases, both of which will 
	make further use of the KMT coupling.  Two relevant functionals of 
	random walk will correspond to two functionals of Brownian motion, 
	and we will need to show that the expectations in the random walk 
	case are asymptotically equivalent to those in the Brownian motion case.  
	For this, we require a few results to show that these functionals are 
	sufficiently well-behaved.
	
	\subsection{Two Brownian functionals}
	
	The functionals of {interest} are \begin{align*}
	X_M &:= \frac{1}{\int_0^M \exp(B_s)\,ds} \\
	Y_M &:= \frac{B_M^-}{\int_0^M \exp(2 B_s)\,ds} 
	\end{align*}
	where we use the notation $Z^-:= \max\{-Z,0\}$.  
	
	\begin{lemma}\label{lem:UI-1}
		The family of variables $\{\frac{X_M}{\E X_M} \}_{M \geq 1}$ is 
		uniformly integrable.
	\end{lemma}  
	
	\begin{proof}
		{ We first claim that $\E[X_M] = \Theta\left( M^{-1/2} \right)$.  Indeed, 
			by Brownian scaling we have $$\E X_M = \E \left(\int_0^M \exp(B_s)\,ds \right)^{-1} = \frac{1}{4}\E \left(\int_0^{M/4}\exp(2 B_t)\,dt \right)^{-1} \sim \frac{1}{4}\cdot \sqrt{\frac{8}{\pi M}} $$
			where the asymptotic relation is by Proposition~\ref{prop:MY}.}
		
		We claim that there exists a universal $C > 0$ so that 
		$\P[X_M \geq t] \leq C t^{-2} M^{-1/2}$.  For $t \leq M^{1/3}$, 
		note that $\P[X_M \geq t]$ is at most the probability that both 
		$\min_{0 \leq { s} \leq 1} B_{ s} \leq -\log(t)$ and that after first 
		hitting $-\log(t)$, $B_s$ does not spend more than $1$ unit of time 
		above $-\log(t)$; { this is because if $\min_{0 \leq s \leq 1} B_s > - \log t$
			then 
			$$ X_M \leq \frac{1}{\int_0^1 \exp(B_s)\,ds} < \frac{1}{\int_0^1 \exp( -\log t) \,ds} = t\,.$$
			
		}
		
		Let $\tau = \inf\{s \in [0,1] : B_s = -\log(t)  \}$.  
		The reflection principle gives 
		$$\P[ \inf_{s \in [0,1]} B_s \leq -\log(t) ] = 
		2 \P[B_1 \leq - \log(t)] \leq C \exp\left(-\frac{1}{2}(\log(t))^2 
		\right)\,.$$
		
		Conditioned on $\tau < \infty$, note that the probability $B_s$ 
		does not go above $-\log(t)$ for more than $1$ unit of time is 
		$\frac{2}{\pi}\arcsin(\sqrt{1/{(M-\tau)}}) = \Theta(M^{-1/2})$ by 
		the strong Markov property together with L\'evy's $\arcsin$ law. 
		{For all $t$ we have}
		$$\P[X_M \geq t] \leq C\exp\left(-\frac{1}{2}(\log(t))^2 \right)
		\frac{1}{\sqrt{M}}\,.$$
		{ For all $t \geq 1$, there is some constant $C'> 0$ so that $$\exp\left(-\frac{1}{2}(\log t)^2\right) \leq C t^{-2}$$
			and so for $t \geq 1$ we have $$\P[X_m \geq t] = O(t^{-2} M^{-1/2})\,.$$
			
			For $K$ large enough, we may then bound $$\E\left[ \frac{X_M}{\E X_M} \one_{X_M/\E X_M \geq K}\right] 
			\leq C \sqrt{M} \int_{K }^\infty t^{-2}M^{-1/2}\,dt \leq C/K\,.$$
			Taking $K \to \infty$ completes the proof.
		}
	\end{proof}
	
	We find the asymptotics of the first two moments of the other functional; {this calculation-heavy proof is done in Appendix \ref{app:better-proofs}.}
	
	\begin{lemma} \label{lem:moment-limit}
		As $M \to \infty$ we have $ \E Y_M \to 1$ and $\E Y_M^2 \sim  4 \sqrt{2/\pi} \sqrt{M}$.
	\end{lemma}

	Lastly, we show that the expectation of $Y_M$ is not dominated by the contribution when $Y_M$ is much larger than $\sqrt{M}$. {Again, we defer the proof to Appendix \ref{app:better-proofs}.}
	
	\begin{lemma}\label{lem:expectation-functional}
		Let $Y_M$ denote the random variable $\frac{B_M^-}{\int_0^M \exp(2B_s)\,ds}$.  Then for all events $E$ with $\P[E] \leq M^{-1/2 - \eps}$ for some $\eps > 0$ we have $\E[Y_M \one_E] = o(1)$ as $M \to \infty$.
	\end{lemma}

	\subsection{Medium-sized case: $N^{-1/2} \ll \delta \ll 1 / \log N$.} \label{sec:medium}
	
	\begin{lemma}\label{lem:large-delta-sandwich}
		There exists a constant $C'$ so that for $\delta$ in ${[0,1-\ee]}$ we have	
		$$e^{-C' \delta \log(N)} 
		\E \frac{\one_{E_n^c}}{N\int_0^1 \exp\left( \sqrt{N} \delta' B_t / \sqrt{2} 
			\right) \,dt} - 1/N^2  \leq  \P_N(\win(\delta)) 
		\leq e^{C' \delta \log(N)} 
		\E \frac{\one_{E_n^c}}{N\int_0^1 \exp\left( \sqrt{N}\delta' B_t / \sqrt{2} 
			\right) \,dt} + 1/N^2 \,.$$
		where $E_N$ is an event with $\P(E_N) \leq N^{-2}$.
	\end{lemma}
	\begin{proof}
		For $t \in [0,N]$, define ${\St_t = \St_{\lfloor t \rfloor}}$.  Then note that 
		$$\sum_{{j =0 }}^{N-1} \exp(\St_j) = \int_{0}^N \exp(\St_t)\,dt\,. $$

		{ Let $$p(N) = Q_1(\tau_N < \tau_0)	 = \frac{1}{\sum_{k = 0}^{N-1}\exp(S_k) }$$ and recall that $\P_N(\win(\delta)) = \E p(N)$.} By \eqref{eq:approx} and \eqref{eq:delta}, we have 
		$$ \exp(-C_\ee \delta) \leq p(N) \cdot \left(\int_{0}^N \exp(\St_t) 
		\,dt \right) \leq \exp(C_\ee \delta)\,.$$
		
		Note that $\St_k$ is a random walk whose increments have variance 
		$\frac{(\delta')^2}{2}$.  By Lemma~\ref{lem:KMT}, there exists a 
		coupling so that 
		$$\P \left[\sup_{t \in [0,N]} \left| \frac{\sqrt{2}}{\delta'} 
		\St_t - B_t\right| \geq C \log(N)  \right] \leq \frac{1}{N^2}\,.$$
		
		Letting $E_N$ denote the event on the left-hand side, conditioned on 
		the event $E_N$, we have 
		$$ \exp\left(-C_\ee \delta - \frac{C \delta'}{\sqrt{2}}\log(N)\right)
		\leq p(N) \cdot \left(\int_{0}^N \exp(\delta' B_t / \sqrt{2}) \,dt \right) 
		\leq \exp\left(C_\ee \delta + \frac{C \delta'}{\sqrt{2}}\log(N)\right)\,.$$
		
		Since $\delta' \leq C_\ee' \delta$ for some constant $C_\ee'$, 
		we can find a new constant $C''$ so that 
		$$ \exp\left(-C'' \delta\log(N)\right)\leq p(N) \cdot 
		\left(\int_{0}^N \exp(\delta' B_t / \sqrt{2}) \,dt \right) 
		\leq \exp\left(C'' \delta \log(N)\right)$$
		
		conditioned on $E_N$.  Because $p(N) \leq 1$, and $\P[E_N] \leq 
		\frac{1}{N^2}$, the lemma follows  taking expectations 
		and Brownian scaling.
	\end{proof}
	
	\noindent{\sc Proof of Theorem~\ref{th:better} for medium $\delta$}.
	Because $\delta \log(N) \to 0$, Lemmas \ref{lem:KMT} and \ref{lem:UI-1} imply 
	$$\E[p(N)] \sim \frac{1}{N} m_{\sqrt{N} \delta' / \sqrt{2}}(1)$$
	{ where the uniform integrability guaranteed by Lemma \ref{lem:UI-1}
		implies that we may ignore the $\one_{E_n^c}$ term}
	Applying Lemma \ref{lem:MY rescaled} completes the proof.
	$\Cox$
	
	\subsection{Large case: $\delta = o(1 / (\log N)^{\ee})$}
	
	Let $r > 6$ be a real parameter { to be chosen later and} set $T := \lceil \delta^{-r} \rceil$.  
	Note that {$\exp (\delta^{-r}) \gg \exp( (\log N)^{r \eps})$} grows faster than any polynomial in $N$ once
	$r \ee > 1$.  Also, we may assume that $\delta^{-r} = o(N^s)$ for any
	positive $s$ because the medium case already covers the regime 
	$\delta \leq (\log N)^{-2}$, say, and in the complement of this case, 
	certainly any negative power of $\delta$ grows more slowly than
	any power of $N$.  { In order to handle the case at hand, we first show that the main contribution to $\P_N(G(\delta))$ is  the first $T$ steps of the walk.  This is handled in Lemma \ref{lem:main}.  Analyzing the resulting functional of the random walk up to $T$ will be done in a similar manner to the ``medium-sized case'': the KMT coupling will be employed to compare the random walk to a Brownian motion, and the corresponding functionals of Brownian motion will be the same as those appearing in Lemmas \ref{lem:moment-limit} and \ref{lem:expectation-functional}.}
	
	Recalling the process
	$\{ \St_k \}$  Section~\ref{sec:scaling}, we denote
	$Z := \St_T$ and $A := \sum_{k=0}^{T-1} \exp (\St_k)$.  For a real number or random variable $X$, we use the notation $X^+ := \max(X,0)$ for the positive part of $X$ and {$X^- := \max(-X,0)$} for the negative part of $X$.
	
	\begin{lemma} \label{lem:main} { Let $s > \eps^{-1}$ where we recall $\delta = o(1/(\log N)^\eps)$.  Then}
		$$\P_N (G(\delta)) =(1+o(1)) \frac{2}{\delta'\sqrt{\pi N}} \left(\E \frac{Z^-}{A} + O\left(\E \frac{\delta^{-s}}{A} \right)\right) + O(e^{-\delta^{-s}}) \, .$$
	\end{lemma}
	
	\begin{proof}
		Recall ~\eqref{eq:delta} that $\{ \St_n \}$ is uniformly close to
		$\{ S_n \}$ as $\delta \to 0$, therefore~\eqref{eq:approx} implies
		$\P_N (G(\delta)) \sim \E (A^{-1})$
		which we will use instead of~\eqref{eq:approx}.
		
		We show the asymptotic equality in the statement of the theorem
		as two inequalities:
		\begin{align} 
		\P_N (G(\delta)) & \geq  (1 + o(1)) \frac{2}{\delta'\sqrt{\pi N}} \left(\E \frac{Z^-}{A} + O\left(\E \frac{\delta^{-s}}{A} \right)\right) 
		\label{eq:LB} \, ; \\
		\P_N (G(\delta)) & \leq  (1 + o(1)) \frac{2}{\delta'\sqrt{\pi N}} \left(\E \frac{Z^-}{A} + O\left(\E \frac{\delta^{-s}}{A} \right)\right) + e^{-\delta^{-s}}
		\label{eq:UB} \, .
		\end{align}
		Choose $r$ so that $r/2 - 1 > s > \ee^{-1}$.  Let $G'$ denote 
		the event $\{ \max_{T \leq m \leq N} \St_m \leq -\delta^{-s} \}$.
		We condition on $A$ and $Z$. {
			\begin{align} 
			\P_N (G(\delta) \,|\,A,Z) &=  (1 + o(1))\E\left[ \frac{1}{\sum_{k = 0}^{N-1}\exp(\St_k) }\,\bigg|\,A,Z\right] \nonumber \\
			&\geq (1 + o(1))\E\left[ \frac{\one_{G'}}{\sum_{k = 0}^{N-1}\exp(\St_k) }\,\bigg|\,A,Z\right] \nonumber \\
			&\geq  \frac{1}{A + N e^{-\delta^{-s}}} \P_N (G' \| A , Z)\,. \label{eq:LB-intermediate}
			\end{align}}
		Since $\St_m / (\delta' / \sqrt{2})$ is a random walk with centered increments of variance $1$, Lemma~\ref{lem:RW} gives $$\P_N(G' \,|\, A,Z) \sim \sqrt{\frac{2}{\pi(N - T)}}\cdot \frac{(-Z - \delta^{-s})^+}{\delta'/\sqrt{2}} \sim \frac{2 (Z + \delta^{-s})^-}{\delta'\sqrt{\pi N}}\,. $$
		
		Combining with \eqref{eq:LB-intermediate} gives $$\P_N(G(\delta) \| A,Z) \geq (1 + o(1))\frac{2}{\delta' \sqrt{\pi N}}\cdot \frac{(Z + \delta^{-s})^-}{A + N e^{-\delta^{-s}}}\,. $$
		
		The quantity $A$ is at least~1, while $N e^{-\delta^{-s}} \to 0$, 
		therefore $A + N e^{-\delta^{-s}} \sim A$.  { Further, write $$(Z + \delta^{-s})^{-} = -Z \one_{Z \leq -\delta^{-s}} - \delta^{-s}\one_{Z \leq -\delta^{-s}} = -Z (\one_{Z \leq 0} - \one_{Z \in [-\delta^{-s},0]}) - \delta^{-s}\one_{Z \leq -\delta^{-s}} = Z^{-} + O(\delta^{-s})\,. $$}
		Taking unconditional expectations now gives~\eqref{eq:LB}.
		
		For the reverse inequality, let $G''$ denote the event that
		$\{ \max_{T \leq m \leq N} \St_m \leq \delta^{-s} \}$.
		On the complement of this event, at least one summand in
		the denominator of~\eqref{eq:approx} is at least ${\exp(\delta^{-s})}$.
		Because $\P_N (G(\delta) \| A , Z)$ is always at most $1/A$, we see that
		\begin{align*}
		\P_N (G(\delta) \| A , Z) &\leq  \frac{1}{A} \P_N (G'' \| A , Z) +  {\frac{1}{A} \P_N ((G'')^c \| A , Z) }\\ 
		&\leq  \frac{1}{A} \P_N (G'' \| A , Z)   + \exp (- \delta^{-s}) \\
		& \leq  (1 + o(1)) \sqrt{\frac{2}{\pi (N-T)}} \frac{(\delta^{-s} - Z)^+}{A \delta'/\sqrt{2}}
		+ e^{-\delta^{-s}} \, .
		\end{align*}
		Again, taking unconditional expectations finishes, yielding~\eqref{eq:UB}.
	\end{proof}

	We are now in a position to apply the KMT coupling to find the expectations that appear in Lemma~\ref{lem:main}.
	
	\begin{lemma} \label{lem:KMT-use}
		\begin{align}
		\E \frac{Z^-}{A} &\sim \delta^2\,;\\
		\E \frac{1}{A} &\sim\frac{\delta}{\sqrt{\pi T}} \,.
		\end{align}
	\end{lemma}

	{ The proof is similar to the results of Section \ref{sec:medium}, and so we defer to the appendix once more.  While the second asymptotic equality follows  the result of Section \ref{sec:medium}, we include a proof as well since the intermediate steps required for the first expression essentially prove it.}

	\noindent{\sc Proof of Theorem~\ref{th:better} for large $\delta$:}  Combining Lemma \ref{lem:main} with Lemma \ref{lem:KMT-use} gives
	\begin{align*}
	P_N(G(\delta)) &= (1 + o(1))\frac{2}{\delta' \sqrt{\pi N}}\left(\E \frac{Z^-}{A} + O\left(\E \frac{\delta^{-s}}{A}\right) \right) + O(e^{-\delta^{-s}}) \\
	&= (1 + o(1))\frac{1}{\delta \sqrt{\pi N}}\left(\delta^2+ O\left(\frac{\delta^{1-s}}{ \sqrt{T}}\right) \right) + O(e^{-\delta^{-s}}) \\
	&\sim \frac{\delta}{\sqrt{\pi N}}
	\end{align*}
	where we used that $T = {\Omega(\delta^{-r})} = \Omega(\delta^{-6})$ to show that $\delta^2 + O(\delta^{1-s} T^{-1/2}) \sim \delta^2$.  $\Cox$
	
	\section{Numerical simulations and further questions, {and biological applications}} \label{sec:data}
	
	To double check the results of Theorem~\ref{th:main}, we simulated
	the process for $N=250$ and $c=2$.  Thus, $\delta = c/\sqrt{N} = 
	2/\sqrt{250} \approx 0.126$.  Theorem~\ref{th:main} predicts that
	as $N \to \infty$, $N \P_N (G(\delta)) \to g(2)$.  Numerically
	evaluating the integral defining $g(2)$ gives approximately $1.516$.
	Our quick and dirty Monte Carlo simulation gives $N \P_N (G(\delta))
	= 1.521 \pm 0.06$.  We could have done more simulations to lower
	the standard error, but in fact because we ran simulations for 
	$N = 10 m$ for every $m \leq 25$, there is already greater accuracy.
	Figure~\ref{fig:line} shows all of these data points, as well as
	similar data for $c=3$ and $1 \leq m \leq 15$ (here $g(3) \approx 1.97$).
	The limits predicted by Theorem~\ref{th:main} are corroborated, or at
	least not contradicted, by the data.
	\begin{figure}[!ht]
		\centering
		\includegraphics[width=4in]{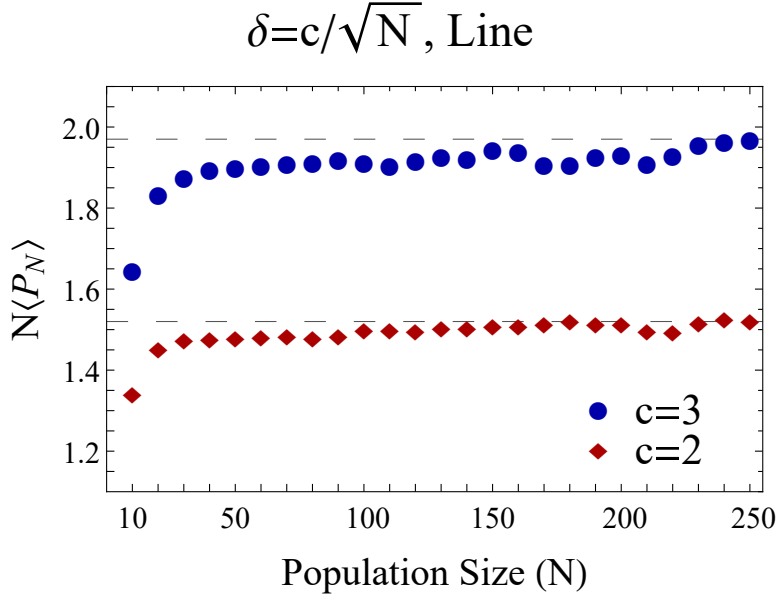}
		\caption{The average mutant fixation probability, $\langle P_N\rangle$ times $N$ as a function of $N$ for the line model. The fitness of both mutants and normals at different locations are {drawn}  the same two-valued distribution function, where values $1-\delta$ and $1 + \delta $ are equally likely. Two different values of $\delta$ are used: $2/\sqrt{N}$ (red points) and $3/\sqrt{N}$ (blue points). The points are based on stochastic simulations and each data point represents the average over $10^6$ independent realizations.}
		\label{fig:line}
	\end{figure}
	
	Next we ran simulations to investigate Conjecture~\ref{conj:1}.
	Recall, the limit is known when $\delta \to 0$ as fast as any power
	$(\log N)^{-\ee}$, whereas this should fail when $\delta$ remains
	constant; the conjecture covers the ground in between, which is
	clearly too slim to distinguish numerically.  The best we could do
	was to hold $\delta$ constant, thus allowing $c := \delta \sqrt{N}$
	to go to infinity. One might 
	expect~\eqref{eq:main} that $N \P_N (G(\delta))$ is well 
	approximated by $g(c)$, leading to 
	\begin{equation} \label{eq:second}
	\sqrt{\pi N} \P_N (G(\delta)) \approx 
	\frac{g(\delta \sqrt{N}) \sqrt{\pi}}{\sqrt{N}} \, ,
	\end{equation}
	which is asymptotic to $\delta$ by~\eqref{eq:c=infty}.
	Indeed, the data (red points in Figure~\ref{fig:infty}) is a
	very good match for~\eqref{eq:second} (the blue curve
	in Figure~\ref{fig:infty}), which can be seen to be asymptotic
	to~$0.2$.
	\begin{figure}[!ht]
		\centering
		\includegraphics[width=3.8in]{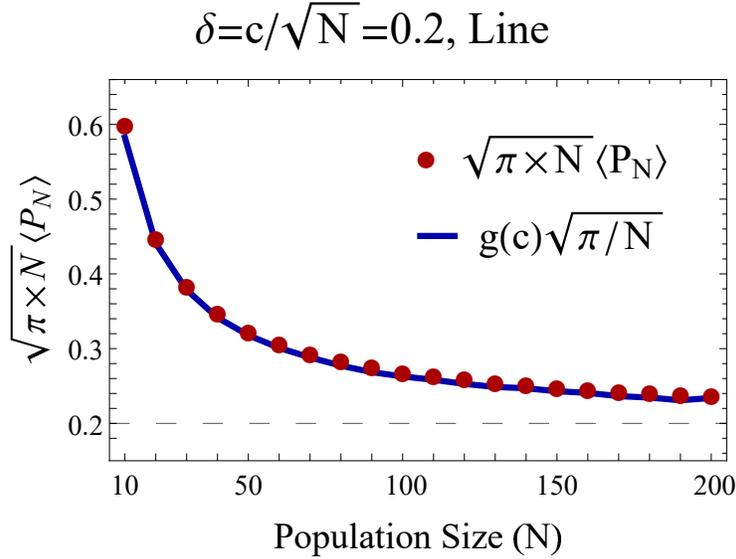}
		\caption{Formula  Eq.\eqref{eq:second}    ($g(\delta\sqrt{N})\sqrt{\pi/N}$, blue curve) is compared with the stochastic simulation results for $\sqrt{\pi\times N}\langle P_N\rangle$ (red points), plotted a functions  of $N$, with  $\delta = 0.2$. }
		\label{fig:infty}
	\end{figure}
	
	Among the open questions on this model, one that looms large is 
	whether these results or something similar can be transferred 
	to the circular model.  Between the line and circle model, neither seems
	inherently more compelling; however the fact that the birth and death
	chain reasoning holds only for the line model has prevented us 
	understanding the situation on any other graphs or initial conditions.
	We enumerate some problems in what we expect to be increasing order of
	difficulty.
	
	\begin{problem} \label{prob:1}
		On a line segment graph, extend the model to the case where the 
		initial configuration is something other than mutants in an
		interval containing an endpoint.
	\end{problem}
	
	\begin{problem} \label{prob:2}
		Extend the analysis to a circle.
	\end{problem}
	
	We were curious whether empirically, the circle appears to behave
	differently  the line.  Figure~\ref{fig:cycle} shows the comparison. 
	It appears that the limiting value of $N \P_N (G(\delta))$ for the 
	circle is just a shade less than for the line.  But also, it appears
	that the value approaches the limit much faster for the circle, and
	perhaps with less sample variance.  On a circle, starting with a single
	mutation, the interval set of mutant sites remains an interval, which
	can now grow and shrink at both ends rather than  just on the right.
	These two growth processes are not independent, but may still be the
	reason we observe faster convergence and lesser variance.
	\begin{figure}[!ht]
		\centering
		\includegraphics[width=4in]{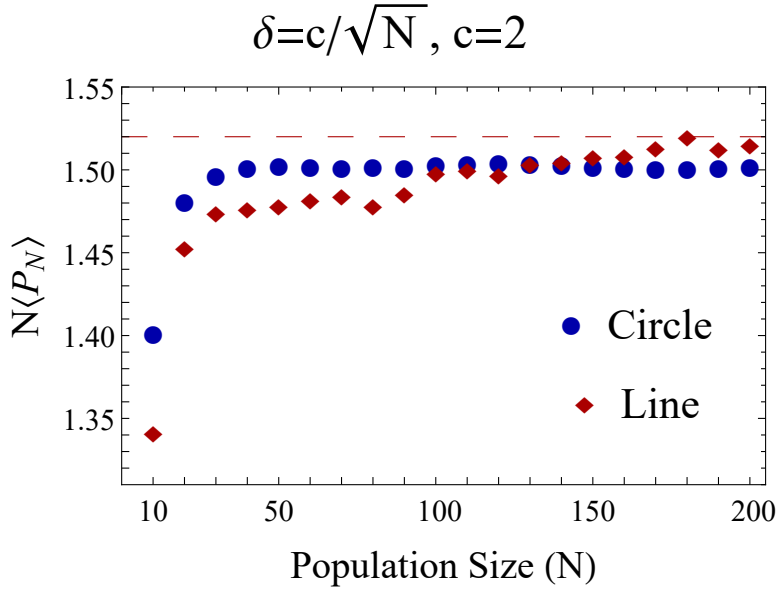}
		\caption{Comparison of the line and the circle model. The quantity $N\langle P_N\rangle$  is plotted as a function of N for the circle model (blue dots) and line model (red dots) with $\delta=2/\sqrt{N}$. For each value of N, the average of $10^6$ random simulations is presented. }
		\label{fig:cycle}
	\end{figure}

	\begin{problem} \label{prob:3}
		Extend the analysis to any graph with a vertex of degree at least~3.  The 
		difficulty here is that the cluster of mutants can become disconnected.
	\end{problem}
	
	{The present study contributes to theoretical understanding of evolutionary processes that have important biomedical applications. The first step in cancer initiation is often a spread and (local) fixation of a neutral mutation, which by itself does not confer an explicit selective advantage to the cell, but serves as a springboard for further transformations. For example, mutations in the so-called tumor suppressor genes drive the progression of many cancers, including colorectal, breast, uterine, ovarian, lung,  head and neck, pancreatic,  and bladder cancer \cite{sherr2004principles, joyce2018cancer}. Tumor suppressor genes are sometimes compared with a brake pedal on a car, as they keep the cell's reproduction in check, preventing it  dividing too quickly. An inactivation of a single copy of a tumor suppressor gene is often considered a ``neural" mutation, because if the second copy is still active, an inactivation of a single copy of the gene does not result in any phenotypic changes. It is only when the second copy of a tumor suppressor gene is inactivated, the cell starts experiencing a selective advantage (because the ``brake" is ``off"). The spread and local fixation of mutants with a single, selectively neutral,  mutation inactivating the first copy of a tumor suppressor gene is the type of problem where our present results can be applied. For example, a plausible scenario for colorectal cancer initiation is fixation of a single-hit mutant, which comes to dominate a local compartment of colonic tissue (called a crypt). This could be followed eventually by the second mutation, which then leads to a local outgrowth and  creation of a ``dysplastic crypt" or a polyp. The first stage (the fixation of neutral, single-hit mutants) has been studied extensively in the context of tumor-suppressor gene inactivation (see e.g. \cite{nowak2002role, komarova2003mutation}) but not in the presence of environmental randomness. Results reported in this paper allow to account for the role of variability in tissue microenvironment, and suggest that single-mutant fixation is more likely than predicted by non-random models. Further models that include more realistic geometries, as well as heterogeneity of cell types (such as stem cells vs differentiated cells) will require further mathematical efforts. The current manuscript lays a foundation for such future efforts.  }

	\section*{Acknowledgments}
	The authors thank Weichen Zhou for comments on a previous draft.  {The authors thank the anonymous referees for comments as well.}
	
	\bibliographystyle{alpha}
	\bibliography{random_cite}
	
	\appendix
	{
		\section{Proofs  Section \ref{sec:brownian}}\label{sec:scaling-proofs}
	}
	\noindent{\sc Proof of Proposition \ref{prop:MY}:} The first two are proved as Matsumoto and 
	Yor~{\cite[Proposition~5.9]{MatsumotoYor2005}}.  
	{We quickly derive the third by integrating}~\eqref{eq:MY1}.  Letting $y := x / \sqrt{t}$,
	\begin{eqnarray*}
		m_2 (t) & = & t^{-1} \int \frac{x e^{-x}}{\sinh (x)} \, dN(0,t) (x) \\
		& = & t^{-1/2} \int \frac{y e^{-\sqrt{t} y}}{\sinh (\sqrt{t} y)}
		\, dN(0,1) (y) \, .
	\end{eqnarray*}
	As $t \to \infty$, the quantity $y e^{-\sqrt{t} y} / \sinh (\sqrt{t} y)$
	converges pointwise to $2|y| \one_{y < 0}$.  Truncating, integrating
	and taking limits gives
	$$t^{1/2} m_2 (t) \to \int_{-\infty}^0 2 |y| \, dN(0,1) (y) 
	= \E |N(0,1)| = \sqrt{\frac{2}{\pi}} \, .$$
	$\Cox$
	
	\noindent{\sc Proof of Lemma \ref{lem:MY rescaled}:}
	Let $f_\alpha (x,t)$ denote the density of $A_\alpha (t)^{-1}$ at $x$.
	Let $W_t := (\alpha / \nu) B_{\nu^2 t / \alpha^2}$.  Then $\{ W_t \}$ 
	is also a standard Brownian motion and $\alpha B_t = {\nu} W_t$.  Hence,
	\begin{eqnarray}
	f_\alpha (x,t) dx & = & \P \left ( \frac{1}{\int_0^t e^{\alpha B_s} \, ds}
	\in [x , x+dx] \right ) \nonumber \\
	& = & \P \left ( \frac{1}{(\nu / \alpha)^2 \int_0^{(\alpha / \nu)^2 t} 
		e^{{\nu} W_u} \, du}
	\in [x , x+dx] \right ) \nonumber \\
	& = & \P \left ( \frac{1}{\int_0^{(\alpha / \nu)^2 t} e^{{\nu} W_u} \, du}
	\in [(\nu / \alpha)^2 x , (\nu / \alpha)^2 x + (\nu / \alpha)^2 dx] 
	\right ) \nonumber \\
	& = & \frac{\nu^2}{\alpha^2} f_\nu  \left ( \frac{\nu^2}{\alpha^2} x , 
	\frac{\alpha^2}{\nu^2} t \right ) \, dx \, .
	\end{eqnarray}
	Consequently, changing variables to $\theta = (\nu^2 / \alpha^2) x$, 
	\begin{eqnarray*}
		m_\alpha (t) & = & \int_{-\infty}^\infty x f_\alpha (x,t) \, dx \\
		& = & \frac{\nu^2}{\alpha^2} \int_{-\infty}^\infty x f_\nu 
		\left ( \frac{\nu^2}{\alpha^2} x , \frac{\alpha^2}{\nu^2} t \right ) 
		\, dx \\
		& = & \frac{\alpha^2}{\nu^2} \int_{-\infty}^\infty 
		\theta f_\nu \left ( \theta , \frac{\alpha^2}{\nu^2} t \right )
		\, d\theta \\
		& = & \frac{\alpha^2}{\nu^2} m_\nu \left ( \frac{\alpha^2}{\nu^2} t
		\right ) \, ,
	\end{eqnarray*}
	proving~\eqref{eq:alpha}.  Set $\nu = 2$ and $t = 1$, plug 
	into~\eqref{eq:MY3} and send $\alpha$ to infinity to obtain
	$$m_\alpha (1) = \frac{\alpha^2}{4} m_2 \left ( \frac{\alpha^2}{4} \right )
	\sim \frac{\alpha^2}{4} \sqrt{\frac{2}{\pi \alpha^2 / 4}} 
	= \frac{\alpha}{\sqrt{2 \pi}} \, ,$$
	proving~\eqref{eq:m}. 
	$\Cox$
	
	{
		\section{Proofs  Section \ref{sec:better}} \label{app:better-proofs}
	}
	
	{\sc Proof of Lemma \ref{lem:RW}:}
	For Brownian motion run to time $M$, the reflection principle gives
	$$\P (\sup_{0 \leq t \leq M} B_t \leq u) = 1 - 2 \P_0 (B_M \geq u)$$
	which is asymptotic to $(2/\pi)^{1/2} u M^{-1/2}$ uniformly as $u$
	varies over the $(0,M^{1/2 - \ee}]$ for any $\ee \in (0,1/2)$.
	Pick $\alpha > 1/2$.  By~\eqref{eq:stretched}, one then has
	$$\sqrt{\frac{2}{\pi}} \frac{u - C_\alpha \log M}{M^{1/2}} - M^{-\alpha}
	\leq \P (\max_{1 \leq j \leq M} S_j \leq u) 
	\leq \sqrt{\frac{2}{\pi}} \frac{u + C_\alpha \log M}{M^{1/2}} 
	+ M^{-\alpha} \, .$$
	$\Cox$
	
	{\sc Proof of Lemma \ref{lem:moment-limit}:}
	By \eqref{eq:MY1}, we compute \begin{align*}
	\E \frac{B_M^-}{\int_0^M \exp(2 B_s)\,ds} &= \int_{\R} \frac{-x^2 e^{-x} \one_{x \leq 0}}{M \sinh(x)} \,dN(0,M)(x) \\
	&=\int_{\R} \frac{-y^2 e^{-\sqrt{M}y} \one_{y \leq 0}}{\sinh(\sqrt{M}y)} \,dN(0,1)(y)\,.
	\end{align*}
	As $M \to \infty$, $e^{-\sqrt{M}y}/\sinh(\sqrt{M}y) \to -2\cdot  \one_{y < 0}$; truncating, integrating and taking limits then gives $$\E \frac{B_M^-}{\int_0^M \exp(2 B_s)\,ds} \to \int_{- \infty}^0 2 y^2 dN(0,1)(y) = 1\,.$$
	
	By differentiating Equation (5.7) of~\cite{MatsumotoYor2005}  
	twice with respect to $\lambda$, we have that 
	$$\E\left[\left(\int_0^M \exp(2 B_s)\,ds\right)^{-2} 
	\, \bigg|\, B_M = x \right] 
	= \frac{e^{-2x}\left(x^2 \sinh(x) + M x \cosh(x) - M \sinh(x) \right)}
	{M^2 \sinh(x)^3}\,.$$
	This implies that 
	\begin{align*}
	\frac{\E[Y_M^2]}{\sqrt{M}} &= \int_\R \frac{e^{-2x}\one_{x < 0}(x^4\sinh(x) + Mx^3 \cosh(x) - Mx^2 \sinh(x)) }{M^{5/2} \sinh(x)^3}\,dN(0,M)(x) \\
	&= \int_\R \frac{e^{-2y \sqrt{M}}\one_{y < 0}(M^2y^4\sinh(\sqrt{M}y) + M^{5/2}y^3 \cosh(\sqrt{M}y) - M^2y^2 \sinh(\sqrt{M}y)) }{M^{5/2} \sinh(\sqrt{M}y)^3}\,dN(0,1)(y)\,.
	\end{align*}
	The integrand converges to $-4 y^3 \one_{y < 0}$ as $M \to \infty$; truncating, integrating and taking limits shows $$\frac{\E[Y_M^2]}{\sqrt{M}} \to \int_\R -4 y^3 \one_{y < 0} dN(0,1)(y) = 4 \sqrt{\frac{2}{\pi}}\,.$$	
	$\Cox$
	
	{\sc Proof of Lemma \ref{lem:expectation-functional}:}
	Using Lemma~\ref{lem:moment-limit} together with Chebyshev's inequality, we see that $$\P[Y_M \geq t] \leq C \sqrt{M} t^{-2}$$ for some constant $C$.  This implies that $$\E[Y_M \one_{Y_M \geq M^{1/2 + \eps/2} }] = \int_{M^{1/2 + \eps/2}}^\infty \P[Y_M \geq t] \,dt \leq C M^{-\eps/2} = o(1)\,. $$
	
	We then may write $$\E[Y_M \one_{E}] = \E[Y_M \one_{E} \one_{Y_M \geq M^{1/2 + \eps/2} }] + \E[Y_M \one_{E} \one_{Y_M < M^{1/2 + \eps/2} }] \leq o(1) + M^{1/2 + \eps/2} \P[E] \to 0\,.$$ 
	$\Cox$
	
	{\sc Proof of Lemma \ref{lem:KMT-use}:}
	By Lemma~\ref{lem:KMT}, there exists a coupling of $\left\{\frac{\St_t}{\delta'/ \sqrt{2}} \right\}_{0 \leq t \leq T}$ and $\{B_t\}_{0 \leq t \leq T}$ so that 
	$$\P \left[  \sup_{t \in [0,T]} \left|\frac{\St_t}{\delta'/\sqrt{2}} - B_t\right| \geq C \log(T)  \right] \leq \frac{1}{T^2}\,.$$
	Let $E$ denote the event in the above probability; conditioned on $E^c$, we have
	$$A  = \int_0^T \exp(\St_s)\,ds = \int_0^T \exp(\delta' B_s/\sqrt{2} + O(\delta' \log(T)))\,ds \sim \int_0^T \exp(\delta' B_s / \sqrt{2})\,ds$$
	where the last asymptotic equality follows  $\delta' \log (T) = O(\delta \log \delta) \to 0$.  This means that 
	\begin{equation}\label{eq:A-moment}\E[A^{-1}] = \E[A^{-1} \one_E] + \E[A^{-1} \one_{E^c}] = O(T^{-2}) + (1+o(1))\E \frac{\one_{E^c}}{\int_0^T \exp(\delta' B_s /\sqrt{2}) \,ds}\,.\end{equation}
	
	By Brownian scaling, 
	$$\int_0^T \exp\left(\delta' B_s /\sqrt{2}  \right)\,ds = \int_0^T \exp\left(2 ( B_{ (\delta')^2 s / 2^{3}} )  \right)\,ds = 2^{3}(\delta')^{-2}\int_0^{T(\delta')^2 / 2^{3} } \exp(2 B_s)\,ds $$
	{ since $\delta' \sim 2\delta$.}
	
	By assumption, $T (\delta')^2 \to \infty$; further, the uniform integrability statement in Lemma \ref{lem:UI-1} shows that $$\E\left[\left(\int_0^{T(\delta')^2 / 2^{3} } \exp(2 B_s)\,ds\right)^{-1} \one_{E^c}\right] \sim \E\left[\left(\int_0^{T(\delta')^2 / 2^{3} } \exp(2 B_s)\,ds\right)^{-1} \right]$$ since the variables $\one_{E} X_M / \E [X_M]$ converge almost surely to $0$ {(because $\P(E) \to 0$)} and are uniformly integrable.  
	From here, \eqref{eq:A-moment} then gives $${\E[A^{-1}] = O(T^{-2}) + (1 + o(1))\cdot\frac{(\delta')^2}{2^3} \cdot \sqrt{\frac{2}{\pi T}} \cdot \frac{2^{3/2}}{\delta'} = O(T^{-2}) + (1 + o(1))\frac{2^2}{\delta'}\frac{1}{\sqrt{\pi T}} \sim \frac{\delta}{\sqrt{\pi T}}\,.}$$
	
	Similarly, \begin{align*}\E[Z^- / A] &= \E\left[ \frac{Z^-}{A} \cdot \one_{E}  \right] + \E\left[ \frac{Z^-}{A} \cdot \one_{E^c}  \right] \\
	&= O(T^{-2}) + \frac{\delta'}{\sqrt{2}}\E\left[ \frac{B_T^- \one_{E^c}}{\int_0^T \exp\left(\delta' B_s /\sqrt{2} \right)\,ds }\right] + O\left( \delta \log(T) \E\left[\left(\int_0^T \exp(\delta' B_s/ \sqrt{2})\,ds\right)^{-1} \right] \right) \\
	&= O\left(\frac{\log(T)}{ \sqrt{T}}\right) + \frac{\delta'}{\sqrt{2}}\E\left[ \frac{B_T^- \one_{E^c}}{\int_0^T \exp\left(\delta' B_s /\sqrt{2} \right)\,ds }\right]\,. \end{align*}
	
	Using { the same Brownian scaling as in below \eqref{eq:A-moment}}, note  \begin{align*}
	\frac{B_T^-}{\int_0^T \exp(\delta' B_s/\sqrt{2})\,ds} &\overset{d}= \frac{\delta'}{2^{3/2}}\frac{B_{(2^{3/2}/\delta')^2 T}^- }{\int_0^{(2^{3/2}/\delta')^2 T} \exp(2 B_{s}) \,ds } = \frac{\delta'}{2^{3/2}} Y_M
	\end{align*}
	
	where we set $M = (2^{3/2} /\delta')^2 T$.  By Lemma~\ref{lem:expectation-functional} together with Lemma~\ref{lem:moment-limit}, we have $$\E[Y_M \one_{E^c}]= \E[Y_M] + o(1) = 1 + o(1)\,. $$
	
	Combining the above equalities provides $$\E[ Z^- / A] = O \left(\frac{\log(T)}{ \sqrt{T}}\right) + (1 + o(1))\cdot \frac{\delta'}{\sqrt{2}}\cdot\frac{\delta'}{2^{3/2}}  \sim \delta^2\,.$$ 
	$\Cox$
	
\end{document}